\documentclass[11 pt]{amsart}
\usepackage{amsfonts}
\usepackage{amsmath,amscd}
\usepackage{fullpage}
\usepackage{amssymb}
\usepackage{centernot} 
\usepackage{enumerate} 
\usepackage{tikz-cd}
\usepackage{parskip}
\usepackage{pb-diagram} 
\usepackage{mathrsfs}
\usepackage[OT2,T1]{fontenc}
\usepackage{seqsplit}
\usepackage{color}
\usepackage{array}
\usepackage{verbatim}
\usepackage{url}

\DeclareSymbolFont{cyrletters}{OT2}{wncyr}{m}{n}
\DeclareMathSymbol{\Sha}{\mathalpha}{cyrletters}{"58}

\definecolor{refkey}{rgb}{1,1,1}
\definecolor{labelkey}{rgb}{1,1,1}
\definecolor{cite}{rgb}{0.9451,0.2706,0.4941}
\definecolor{ruri}{rgb}{0.0078,0.4022,0.8010}

\usepackage[%
bookmarks=true,bookmarksnumbered=true,%
colorlinks=true,linkcolor=ruri,citecolor=red%
 ]{hyperref}

\makeindex 

\def\F{{\rm \mathbb{F}}}
\def\Z{{\rm \mathbb{Z}}}

\def\Q{{\rm \mathbb{Q}}}

\def\C{{\rm \mathbb{C}}}

\def\R{{\rm \mathbb{R}}}
\def\T{{\rm \mathbb{T}}}

\def\p{{\rm \mathfrak{p}}}

\def\O{{\rm \mathcal{O}}}

\def\Jac{{\rm Jac}}

\def\ord{{\rm ord}}

\def\avg{{\rm avg}}

\def\Aut{{\rm Aut}}

\def\Pic{{\rm Pic}}

\def\e{{\rm \epsilon}}

\def\GL{{\rm GL}}
\def\Gal{{\rm Gal}}

\def\rk{{\rm rk}}

\def\Hom{{\rm Hom}}
\def\End{{\rm End}}
\def\Sel{{\rm Sel}}

\numberwithin{equation}{section}

\newtheorem{theorem}{Theorem}[section]
\newtheorem{lemma}[theorem]{Lemma}

\newtheorem{assumption}[theorem]{Assumption}
\newtheorem{question}{Question}
\newtheorem{remark}[theorem]{Remark}
\newtheorem{definition}[theorem]{Definition}
\newtheorem{example}[theorem]{Example}
\newtheorem{conjecture}[theorem]{Conjecture}
\newtheorem{corollary}[theorem]{Corollary}
\newtheorem{proposition}[theorem]{Proposition}

\usepackage{tikz}

\def\shownotes{\def\inline##1##2##3{ \begin{adjustwidth}{3mm}{7mm}\mbox{}\par \noindent
{\color{##1}\hspace{-1.9cm}{\large ##2}\vspace{-\baselineskip}\\##3}
\newline\end{adjustwidth}} \def\inlinewide##1##2##3{ \begin{adjustwidth}{0mm}{0cm}\mbox{}\par \noindent
{\color{##1}\hspace{-1.6cm}{\large ##2}\vspace{-\baselineskip}\\##3}
\newline\end{adjustwidth}}  \def\marg##1##2##3{\marginnote{\color{##1}{\large ##2}\\{\small ##3}}[-.8cm]}}

\shownotes

\begin{document}
\setlength{\parskip}{2pt} 
\setlength{\parindent}{8pt}
\title{Quadratic twists of abelian varieties with \\real multiplication }
\author{Ari Shnidman}
\address{Department of Mathematics, Boston College, Chestnut Hill, MA 02467} 
\email{shnidman@bc.edu}

\maketitle

\begin{abstract}
Let $F$ be a totally real number field and $A/F$ a principally polarized abelian variety with real multiplication by the ring of integers $\O$ of a totally real field.  Assuming $A$ admits an $\O$-linear 3-isogeny over $F$, we prove that a positive proportion of the quadratic twists $A_d$ have rank 0.  We also prove that a positive proportion of $A_d$ have rank $\dim A$, assuming the groups $\Sha(A_d)$ are finite.  If $A$ is the Jacobian of a hyperelliptic curve $C$, we deduce that a positive proportion of twists $C_d$ have no rational points other than those fixed by the hyperelliptic involution. 
            
\end{abstract}

\makeatletter
\makeatother

\section{Introduction}

As part of his study of the modular Jacobians $J_0(p)$ \cite{Mazur}, Mazur showed that there are geometrically simple abelian varieties $A/\Q$ of arbitrarily large dimension having finite Mordell-Weil group.  The existence of rank 0 abelian varieties also follows from results of Gross-Zagier, Kolyvagin-Logachev, and Bump-Friedberg-Hoffstein \cite{BFH, gz, KL}.  These examples are obtained by considering quadratic twists of simple quotients of $J_0(N)$, for any $N$.           

Of course, one expects much more than the mere existence of such abelian varieties.  Presumably, a significant  proportion of all abelian varieties should have rank 0.  Unfortunately, it is hard to prove quantitative results in this direction.  Indeed, it was only recently that Bhargava and Shankar proved that a positive proportion of elliptic curves over $\Q$ have rank 0 \cite{BS2}.  Geometry-of-numbers methods have since been deployed to study various algebraic families of abelian varieties of higher dimension  \cite{BG2, BGW, PoonenStoll2,Ananth},  
yet the following question still seems to be open.

\begin{question}\label{question}
Is there a non-trivial, algebraic family of abelian varieties of dimension $g> 1$ over $\Q$ for which a positive proportion of members are geometrically simple and have rank $0?$
\end{question}

In this note, we use geometry-of-numbers to provide many such examples.  Our families turn out to be quadratic twist families of simple quotients of $J_0(N)$, but our approach is via Selmer groups and does not use any automorphic input.
 
In fact, recent work of the author, with Bhargava, Klagsbrun, and Lemke Oliver, already answered Question \ref{question} over certain number fields.  For instance, let $J$ be the Jacobian of the genus 3 curve $y^3 = x^4 -x$, and let $L$ be any number field containing the cyclotomic field $\Q(\zeta_9)$.  For each $d \in L^\times/L^{\times2}$, let $J_d$ be the $d$th quadratic twist of $J$.  By studying the 3-Selmer group of $J_d$, we showed that at least $50\%$ of the twists $J_d$ have rank 0 \cite[Thm.\ 11.5]{BKLS}.  One can construct similar examples using other abelian varieties with complex multiplication (CM).  Indeed, our proof leverages the fact that $J$ has an endomorphism of degree 3 over $L$.  Since an abelian variety cannot have CM over $\Q$, this approach does not readily give examples of the desired type over $\Q$.  

This paper considers a different class of abelian varieties with extra endomorphisms: those with {\it real multiplication}  These have the advantage that their extra endomorphisms may be defined over $\Q$.  The disadvantage is that they do not admit  endomorphism of degree 3.  To compensate for this, we will impose a bit of level structure.   

\subsection{Real multiplication and $\O$-linear isogenies}
Let $F$ be a number field, and $(A, \lambda)$ a polarized abelian variety over $F$ of dimension $g$.  Let $K$ be a totally real number field of degree $g$ over $\Q$, and let $\O$ be a subring of finite index in the ring of integers $\O_K$.  We say $(A, \lambda)$ has {\it real multiplication} (RM) by $\O$ if there is an embedding $\iota \colon \O \hookrightarrow \End_F(A)$ such that the polarization $\lambda \colon A \to \hat A$ is $\O$-linear with respect to the induced actions of $\O$ on $A$ and $\hat A$.  These abelian varieties are sometimes said to be of {\it $\GL_2$-type}. 

Any elliptic curve has RM by $\Z$.  But among abelian varieties of dimension greater than 1, those with RM are quite special.  Special in the sense that they lie on proper closed subvarieties of the corresponding moduli space, but also because their arithmetic parallels the arithmetic of elliptic curves in certain respects.  For example, if $F = \Q$, then it follows from a result of Ribet \cite{Ribet} that $A$ is a quotient of a modular Jacobian $J_0(N)$ for some $N$.  Conversely, the simple quotients of $J_0(N)$ have RM, so this gives a steady source of examples of RM abelian varieties.

We will consider abelian varieties $A$ with RM by $\O$ and with an additional bit of level structure.  We say $A$ {\it admits an $\O$-linear $3$-isogeny} if there is another abelian variety $B$ with RM by $\O$, and an $\O$-equivariant 3-isogeny $\phi \colon A \to B$ over $F$.  The kernel of $\phi$ is then an $\O$-module (scheme) which is annihilated by an ideal $I$ of index 3 in $\O$.  We call $I$ the {\it kernel ideal} of $\phi$.\footnote{This terminology is used in \cite{Waterhouse}, but with a slightly different meaning.}  

\subsection{Main results}     
For each $d \in F^\times/F^{\times 2}$, we write $A_d$ for the quadratic twist by the character $\chi_d \colon \Gal(\bar F/F) \to \{\pm 1\}$ corresponding to the extension $F(\sqrt{d})/F$.  There is a natural {\it height} function on $F^\times/F^{\times2}$, given by 
\[H(d) = \prod_{\p \colon \ord_\p(d) \mbox{\tiny{ odd}}} N(\p),\] 
the product being over finite places $\p$ of $F$.  This gives an ordering of the quadratic twist family.  

Our first result answers Question \ref{question} affirmatively, and provides a large class of examples.  
\begin{theorem}\label{main intro}
Let $F$ be a totally real number field.  Let $(A, \lambda)$ be a polarized abelian variety over $F$ with $\mathrm{RM}$ by $\O$, and with $3 \nmid \deg(\lambda)$.  Suppose $A$ admits an $\O$-linear $3$-isogeny 
whose kernel ideal $I$ is an invertible $\O$-module. 
Then a positive proportion of  twists $A_d$, $d \in F^\times/F^{\times2}$, have rank $0$.
\end{theorem}

In Section \ref{examples} we give examples of $A$ satisfying the hypotheses of Theorem \ref{main intro}.  For now, we simply note that $I$ is automatically invertible if $I$ is principal or if the index $[\O_K \colon \O]$ is not divisible by 3.  Moreover, in the latter case, $A$ necessarily admits a polarization $\lambda$ of degree prime to $3$.  

Theorem \ref{main intro} gives the first progress in dimension $g > 1$ towards the following conjecture, which is the natural extension of Goldfeld's conjecture \cite{goldfeld} for quadratic twists of elliptic curves over $\Q$.     

\begin{conjecture}\label{goldfeld+}  
Let $A/\Q$ be a simple quotient of $J_0(N)$ of dimension $g$.  Then $50\%$ of twists $A_d$ have rank $0$, and $50\%$ of twists $A_d$ have rank $g$.  
\end{conjecture}

Note that the rank of $A(\Q)$ is a multiple of $g$, since $A$ has RM.  On the analytic side, $L(A,s)$ is the product of $g$ automorphic $L$-functions, all with the same root number.  Thus, Goldfeld's minimalist philosophy and the conjecture of Birch and Swinnerton-Dyer lead directly to conjecture \ref{goldfeld+}. 

Our next result concerns the quadratic twists with rank $g$, the smallest possible non-zero rank.  It is conditional on the finiteness of the Tate-Shafarevich group, but we prove an unconditional result on the ranks of certain Selmer groups defined in Section \ref{PTSelmer}.    

\begin{theorem}\label{rank g}
Let $(A,\lambda)$ and $I$ be as in Theorem $\ref{main intro}$, and let $\pi \colon A \to A/A[I]$ be the natural $(3,3)$-isogeny.  Then a positive proportion of  twists $A_d$, for $d \in F^\times/F^{\times2}$, satisfy $\dim_{\F_3} \Sel_\pi(A_d) = 1$.  If $\Sha(A_d)$ is finite for all $d$, then a positive proportion of the twists $A_d$ have rank $g = \dim A$.  
\end{theorem}

In the elliptic curve case $g = 1$, Theorems \ref{main intro} and \ref{rank g} were proven by Bhargava, Klagsbrun, Lemke Oliver, and the author \cite[Thm.\ 1.6]{BKLS}.  And indeed, a key ingredient in the proofs of Theorem \ref{main intro} and \ref{rank g} is the general result  \cite[Thm.\ 1.1]{BKLS}, which determines the average size of the Selmer groups $\Sel_{\phi_d}(A_d)$, where $\phi_d \colon A_d \to A'_d$ is the quadratic twist family of {\it any} 3-isogeny $\phi \colon A \to A'$ of abelian varieties.   
The other two main ingredients are an analysis of Selmer groups in `$\O$-isogeny chains' of RM abelian varieties, and a collection of arithmetic duality theorems.  


We have been careful to allow non-principal polarizations and non-maximal rings $\O$ in order to cover many examples of RM abelian varieties which arise naturally.  One place to look for such examples is the modular Jacobians $J_0(p)$ over $\Q$ of prime level $p$.  If $p \equiv 1 \pmod 9$, then $J_0(p)$ has a point of order 3, and by a result of Emerton, there is at least one optimal quotient $A$ with a point of order 3.  If moreover $p \not\equiv 1\pmod{27}$, then $A$ is unique.    
In Section \ref{examples}, we prove:  
\begin{theorem}\label{mazur}
Suppose $p \equiv 10$ or $19 \pmod{27}$, and let $A$ be the optimal quotient of $J_0(p)$ with a rational point of order $3$.  Assume $A$ has a polarization of degree prime to $3$.  Then a positive proportion of quadratic twists $A_d$, for $d \in \Z$ squarefree, have rank $0$.  If $\Sha(A_d)$ is finite for all $d$, then a positive proportion of twists have rank $\dim A$.      
\end{theorem}    

The proof uses Mazur's analysis of the Eisenstein ideal to verify the invertibility hypothesis in Theorem \ref{main intro}.  The condition on the polarization degrees of $A$ is presumably satisfied for most such optimal quotients; we give some examples in Section \ref{examples}.  It is possible that when $F = \Q$, the results on twists of rank $g$ can be made unconditional by extending the methods of Kriz-Li \cite[Thm.\ 7.1]{krizli} to higher dimensional modular abelian varieties.    

Our final result on ranks is an explicit upper bound on the average rank of $A_d(F)$.   

\begin{theorem}\label{rank bound intro}
Let $(A,\lambda)$ be as in Theorem $\ref{main intro}$, and let $\phi_d \colon A_d \to B_d$ be the corresponding family of $\O$-linear $3$-isogenies.   Then the average rank of the quadratic twists $A_d$, $d \in F^\times/F^{\times 2}$, is at most $\dim A \cdot \avg_d \left(t(\phi_d) + 3^{t(\phi_d)} \right)$, where $t(\phi_d)$ is the absolute log-Selmer ratio of $\phi_d$.      
\end{theorem}

The average here is computed with respect to the height function $H(d)$ on $F^\times/F^{\times2}$.  See Section \ref{proofs} for the definition of the absolute log-Selmer ratio.  If $\dim A = 1$, then this is \cite[Thm.\ 1.4]{BKLS}.

Theorem \ref{rank bound intro} is explicit in the sense that the stated upper bound can be easily computed if one knows the Tamagawa numbers of $A$ and $B$, and all their quadratic twists, along with the root number of $A$.  With this data, one can also give an explicit lower bound on the proportions of rank 0 (resp.\ $\pi$-Selmer rank 1) quadratic twists.  In general, these Tamagawa numbers are not easy to compute when $\dim A > 1$.  

\subsection{Rational points on curves of genus $g \geq 2$} Theorem \ref{main intro} also has consequences for the study of rational points in quadratic twist families of curves $C/F$ of genus $g \geq 2$.  For such families, it is natural to ask about the average size of $\#C_d(F)$, since $\#C_d(F) < \infty$.  The notion of quadratic twist does not make sense for an arbitrary curve of genus $g > 1$, but it does make sense for hyperelliptic curves.  In such families, one expects very few rational points other than the `trivial' rational points, i.e.\ the rational points fixed by the hyperelliptic involution: 

\begin{conjecture}[Granville \cite{Granville}]\label{granville}
If $C$ be a smooth hyperelliptic curve over $\Q$ of genus $g \geq 2$, then for $100\%$ of $d \in \Q^\times/\Q^{\times2}$, the set $C_d(\Q)$ consists only of fixed points for the hyperelliptic involution.      
\end{conjecture}           

More generally, suppose $C$ is a smooth projective curve over $F$ with an involution $\nu$ that induces  $[-1]$ on the Jacobian $J = \Jac(C)$.  Then for each $d \in F^\times/F^{\times2}$, there is a quadratic twist $C_d$ and whose Jacobian is $J_d$.  If $C_d$ has a rational point, then it embeds in $J_d$ via the Abel-Jacobi map.  Moreover, any rational point in $C(F)$ fixed by $\nu$ gives rise to a rational point in $C_d(F)$ which is fixed by the involution $\nu_d$.  Following Conjecture \ref{granville}, one expects $C_d(F)$ to have no rational points other than these fixed points, on average.  This automatically holds if $J_d(F)$ has rank 0 and if $J_d(F)_{\mathrm{tors}} \neq J_d[2](F)$, and the latter condition holds for all but finitely many $d$.  Thus, Theorem \ref{main intro} immediately gives the following partial result towards Conjecture \ref{granville}. 
\begin{theorem}\label{curves}
Let $F$ be a totally real field, and suppose $C_d$ is the family of quadratic twists of a smooth projective curve $C$ over $F$ with involution $\nu$.  Assume $J = \Jac(C)$ has $\mathrm{RM}$ by $\O$ and admits an $\O$-linear $3$-isogeny whose kernel ideal is an invertible $\O$-module.  Then for a positive proportion of $d \in F^\times/F^{\times 2}$, the only rational points on $C_d$ are the fixed points of $\nu_d$.          
\end{theorem}

In Section \ref{examples} we give some genus 2 examples of such curves, with both simple Jacobians and rational fixed points.  This gives progress towards Conjecture \ref{granville} for many curves which were previously inaccessible.  By considering curves with local obstructions or with a map to an elliptic curve, one can find families of curves satisfying the conclusion of Theorem \ref{curves}, but these tricks do not help when $C$ has rational fixed points and simple Jacobian.    

Finally, we note that Granville gave a conditional proof of many cases of Conjecture \ref{granville}, assuming the {\it abc}-conjecture \cite{Granville}.  For strong unconditional results in the family of {\it all} odd degree hyperelliptic curves, see the work of Poonen-Stoll \cite{PoonenStoll2}.  

\subsection{Acknowledgements}
The author thanks Manjul Bhargava, Pete Clark, Victor Flynn, Eyal Goren, Ben Howard, Zev Klagsbrun, Robert Lemke Oliver, Barry Mazur, Bjorn Poonen, Alice Silverberg, Drew Sutherland, and Yuri Zarhin for helpful conversations. 

\section{Abelian varieties with real multiplication}\label{sect: RM}

Let $F$ be a number field.  Also, let $K$ be a totally real number field of degree $g \geq 1$ over $\Q$, and let $\O$ be a subring of finite index in the ring of integers $\O_K$.  

\subsection{Real multiplication and isogenies}
\begin{definition}{\em
A $g$-dimensional abelian variety $A$ over $F$ {\it has real multiplication by $\O$} if there is an algebra embedding $\O \hookrightarrow \End_F(A)$.  
}\end{definition}

Let $A$ be an abelian variety with real multiplication (RM) by $\O$, and fix $\O \hookrightarrow \End_F(A)$, so that we may think of elements of $\O$ as endomorphisms of $A$.  Our goal is to study the ranks of the quadratic twists $A_d$, for $d$ in $F^\times/F^{\times2}$.  Recall that $A_d$ is the twist of $A$ by the character $\chi_d \colon G_F \to \{\pm 1\}$ corresponding to $F(\sqrt{d})$; here $G_F = \Gal(\bar F/F)$.     

\begin{lemma}\label{twisting RM}
The abelian variety $A_d$ has RM by $\O$, for all $d \in F^\times/F^{\times2}$.  More precisely, the embedding $\iota \colon \O \hookrightarrow \End_F(A)$ induces an embedding $\iota_d \colon \O \hookrightarrow \End_F(A_d)$. 
\end{lemma}
 
\begin{proof}
Since the automorphism $-1 \in \Aut_F(A)$ commutes with action of $\O \subset \End_F(A)$, we see that the action of $\O$ on $A \otimes_F F(\sqrt d) \simeq A_d \otimes _F F(\sqrt d)$ descends to $A_d$.    
\end{proof}

Since $\iota$ is fixed, we can safely consider $\alpha \in \O$ as an endomorphism of $A$ and as an endomorphism of $A_d$, without any ambiguity.  

In order to say something about the ranks of the twists $A_d$, we suppose from now on that $A$ admits an $\O$-linear 3-isogeny.  In other words, we assume there exists an abelian variety $B$ which also has RM by $\O$ over $F$ and a 3-isogeny $\phi \colon A \to B$ over $F$ which is $\O$-equivariant.

The kernel $A[\phi]$ is an $\O$-module scheme of order 3, and hence is annihilated by an ideal $I$ in $\O$ of index 3.  We call  $I$ the {\it kernel ideal} of $\phi$.  Since $[K \colon \Q] =  \dim A$, the subgroup $A[I]$ of all $I$-torsion points on $A$ is isomorphic to $(\Z/3\Z)^2$ over $\bar F$, and $A[\phi] \subset A[I]$.  The quotient $C = A/A[I]$ also has RM by $\O$, and the natural $(3,3)$-isogeny $\pi \colon A \to C$ is $\O$-linear as well.  

\begin{remark}\label{3isog}{\em Conversely, if $I \subset \O$ is an ideal of index 3, and if the $(3,3)$-isogeny $ A \to A/A[I]$ factors through a 3-isogeny $\phi \colon A \to B$, then $B$ has RM by $\O$ and $\phi$ is necessarily $\O$-linear with kernel ideal $I$.}
\end{remark}

Let $\phi' \colon B \to C$ be the 3-isogeny over $F$ such that $\pi = \phi' \circ \phi$.  By Lemma \ref{twisting RM}, there are 3-isogenies $\phi_d \colon A_d \to B_d$ and $\phi'_d \colon B_d \to C_d$, such that $\phi'_d \circ \phi_d = \pi_d \in \Hom(A_d,C_d)$, for each $d \in F^\times/F^{\times 2}$.  Note that 
\begin{equation}\label{twisting}
A_d[\phi_d] \simeq A[\phi] \otimes \chi_d \hspace{3mm} \mbox{and} \hspace{3mm} B_d[\phi'_d] \simeq B[\phi'] \otimes \chi_d
\end{equation}
as $G_F[\F_3]$-modules.  

\subsection{Polarizations and duality}
A {\it polarization} is a homomorphism $\lambda \colon A \to \hat A$ over $F$ which, over $\bar F$, takes the form $\phi_L \colon A_{\bar F} \to \hat A_{\bar F}$ for some ample line bundle $L \in \Pic(A_{\bar F})$.  If $\lambda$ is an isomorphism, we call it a {\it principal polarization}.

In order to study the ranks of the twists $A_d$, we further impose:  
\begin{assumption}\label{assump}{\em 
Assume $A$ admits an $\O$-linear 3-isogeny $\phi \colon A \to B$ over $F$, and that 
\begin{enumerate}
\item $F$ is a totally real field.
\item $A$ admits an $\O$-linear polarization $\lambda \colon A \to \hat A$ over $F$ of degree prime to 3.
\item The kernel ideal $I$ of $\phi$ is an invertible $\O$-module.
\end{enumerate}
}
\end{assumption}   

In particular, the polarized abelian variety $(A,\lambda)$ has RM by $\O$, as defined in the introduction.  These assumptions allow us to glean extra information on the group schemes $A[\pi]$, $A[\phi]$ and $B[\phi']$.   
\begin{proposition}\label{selfdual}
The group scheme $A[\pi]$ is self-dual.
\end{proposition}
\begin{proof}
Since $I$ is an invertible $\O$-module, there exists an ideal $J$ of $\O$ which is coprime to $3$ and such that $(\alpha)I = (\beta)J$, for $\alpha, \beta \in \O$.  It follows that $C = A/A[I]$ is also isomorphic to $A/A[J]$.  Since $J$ is prime to 3, we may fix an isogeny $\psi \colon C \to A$ of degree prime to 3 such that the composition $\psi \circ \pi$ lies in $\O \subset \End_F(A)$.  By $\O$-linearity of $\lambda$, the following diagram is commutative: 
\[\begin{tikzcd}
A \arrow{d}{\lambda} \arrow{r}{\pi} & C \arrow{r}{\psi} & A \arrow{d}{\lambda}\\
\widehat A \arrow{r}{\widehat\psi} & \widehat{C} \arrow{r}{\widehat \pi} & \widehat A
\end{tikzcd}\hspace{40pt}%
\]
Since  $\deg(\lambda)$ and $\deg(\hat\psi)$ are prime to 3, the kernel $A[\pi]$ maps isomorphically onto $\hat C[\hat \pi]$, and hence $A[\pi]$ is self-dual. 
\end{proof}

\begin{proposition}\label{kernel duality}
For any $d \in F^\times/F^{\times 2}$, the group schemes $A_d[\phi_d]$ and $B_d[\phi'_d]$ are Cartier dual.   
\end{proposition}

\begin{proof}
We identify group schemes such as $A_d[\phi_d]$ and $B_d[\phi'_d]$ with their corresponding $G_F[\F_3]$-modules.  Note that isomorphism classes of group schemes $H$ of order 3 over $F$ are in bijection with quadratic characters $\chi \colon G_F \to \{\pm 1\}$. Moreover, the Cartier dual of $H$ corresponds to the character $\chi\chi_3$, where $\chi_3$ is the character of $\mu_3$.  Thus by (\ref{twisting}), we may reduce to the case $d = 1$ and drop the subscripts.    

Let $\chi$ and $\chi'$ be the quadratic characters corresponding to $A[\phi]$ and $B[\phi']$.  These are the Jordan-Holder factors of $A[\pi]$.  Since $A[\pi]$ is self-dual, 
we have the equality of multi-sets
\[\{\chi, \chi'\} = \{\chi\chi_3, \chi'\chi_3\},\]
since the right hand side is the set of Jordan-Holder factors of the dual of $A[\pi]$.   As $F$ is totally real, we have $\chi \neq \chi  \chi_3$, and so we must have $\chi = \chi' \chi_3$.  Thus $A[\phi]$ and $B[\phi']$ are Cartier dual.   
\end{proof}

\begin{remark}{\em In fact, Proposition \ref{kernel duality} holds over any $F$ not containing $\sqrt{-3}$.}
\end{remark}

From the proof of the propositions, we obtain the following corollaries.

\begin{corollary}\label{isog isom}
There is an isogeny $A \to \hat C$ of degree prime to $3$, which induces an isomorphism $A[\phi] \to \hat C[\hat\phi']$.  In particular, there is an isomorphism $H^1(G_F,A[\phi]) \simeq H^1(G_F,\hat C[\hat\phi'])$.     
\end{corollary}

\begin{corollary}\label{B selfdual}
The polarization $\lambda$ induces an isogeny $\lambda_B \colon B \to \hat B$ of degree prime to $3$. 
\end{corollary}

\begin{proof}
Using the notation from the proof of Proposition \ref{selfdual}, the composition $\hat\psi \circ \lambda$, sends $A[\phi]$ isomorphically onto $\hat C[\hat\phi']$, and so the composition 
\begin{equation*}
B \simeq A/A[\phi] \to \hat C/\hat C[\hat \phi'] \simeq \hat B,
\end{equation*}  
gives the desired isogeny.
\end{proof}

\begin{corollary}
The isogeny $\lambda_B \colon B \to \hat B$ is equal to its own dual.
\end{corollary}

\begin{proof}
We give an alternate construction of $\lambda_B$.  Let $\alpha = \psi \circ \pi \in \End_F(A)$.  The $\O$-linearity of $\lambda$ exactly means that $\alpha$ is fixed by the Rosati involution on $\End_F(A)$ associated to $\lambda$.  Thus, the homomorphism $\lambda \circ \alpha \colon A \to \hat A$ is symmetric, i.e.\ equal to its own dual.  It follows that over $\bar F$, we have $\lambda \circ \alpha = \phi_L$ for some line bundle $L$ on $A_{\bar F}$.

We claim that $L$ is of the form $\phi^*M$, for some line bundle $M$ on $B_{\bar F}$.  Indeed, by the theory of descent of line bundles on abelian varieties, it is enough to show that the (cyclic) group scheme $A[\phi]$ lies in the kernel of $\lambda \circ \alpha$, which is of course true since $A[\phi] \subset A[\pi]$.  By degree considerations, it follows that the symmetric isogeny $\phi_M \colon B_{\bar F} \to \hat B_{\bar F}$ has degree $\deg(\lambda_B)$.  By construction, we have $\phi_M = \hat\phi^{-1} \lambda \alpha \phi^{-1}$.  Since $\phi$, $\alpha$, and $\lambda$ are all defined over $F$, the map $\phi_M$ descends to a morphism $B \to \hat B$ over $F$, which is the isogeny $\lambda_B$ defined earlier.   
\end{proof}

\begin{remark}\label{pp}{\em 
The symmetric isogeny $\lambda_B \colon B \to \hat B$ is {\it not necessarily} a polarization, even if $\deg(\lambda) = 1$ and $I$ is principal (in which case $\deg(\lambda_B) = 1$).  For example, if $K = \Q(\sqrt{3})$ with $\pi = \sqrt{3}$, then $\lambda_B$ is not a polarization.  On the other hand, if $K = \Q(\sqrt{6})$ with $\pi = 3 + \sqrt{6}$, then $\lambda_B$ {\it is} a (principal) polarization.  The difference is that $3 + \sqrt{6}$ is totally positive, while $\sqrt{3}$ is not.  See \cite[Thm.\ 2.10]{GGR}.   
}
\end{remark}

\subsection{Isogeny chains}\label{chainz}
Let $k \geq 1$ be the order of the class $[I]$ in $\Pic(\O)$.  Then $I^k = \e\O$, for some $\epsilon \in \O$.  We will also regard $\epsilon$ as an endomorphism of $A$.  If $J$ is any $\O$-ideal, then the quotient $A/A[J]$ has RM by $\O$.  We may therefore recursively define 
\[A_1 = A \hspace{2mm} \mbox{ and } \hspace{2mm} A_i = A_{i -1}/A_{i-1}[I] \mbox{ for } i \geq 2.\] 
Note that $C = A_2$ and 
\[A_{k+1} = A/A[I^k] = A/A[\epsilon] \simeq A,\]
so that $A_{k + i} \simeq A_i$ for all $i \geq 1$.    

For $i \geq 1$, we define $\pi_i \colon A_i \to A_{i+1}$ to be the natural $(3,3)$-isogeny with kernel $A_i[I]$.  In particular, $\pi_1 = \pi$ and $\epsilon = \pi_k\pi_{k-1}\cdots \pi_2 \pi_1 \in \End_F(A)$.  We also have $\pi_{k + i} = \pi_i$ for all $i \geq 1$.            

\begin{proposition}\label{isogchain}
For each $i \geq 1$, there exists an isogeny $A_i \to A_{i+1}$ of degree prime to $3$, which maps $A_i[\pi_i]$ isomorphically onto $A_{i+1}[\pi_{i+1}]$.  
\end{proposition}
\begin{proof}
We may immediately reduce to the case $i = 1$, so that $A_2 = C$.  Since $I$ is an invertible $\O$-module, we may choose an invertible $\O$-ideal $J$ which is prime to 3 and in the same ideal class of $I$ in $\Pic(\O)$.  Then $C \simeq A/A[J]$ and the quotients $\eta \colon A \to C = A/A[J]$ and $\tilde \eta \colon C \to A_3 \simeq C/C[J]$ have degrees prime to 3.  The commutativity of the diagram
 \[
  \begin{tikzcd} 
   & A/A[J] \simeq C\arrow{dr}{\pi_2}
     &   \\  
   A \arrow{dr}{\pi_1} \arrow{ur}{\eta}
   & & A/A[IJ]  \simeq A_3 \\
   & A/A[I] \simeq C\arrow{ur}{\tilde \eta}
     & 
  \end{tikzcd}
  \]
shows that $\eta$ sends $A[\pi]$ isomorphically onto $C[\pi_2]$, as desired.
\end{proof}

\begin{corollary}
Each $(3,3)$-isogeny $\pi_i \colon A_i \to A_{i+1}$ factors as a composition of $3$-isogenies, and the endomorphism $\epsilon \colon A \to A$ factors as a composition of $3$-isogenies.  
\end{corollary}
\begin{corollary}\label{selmerchainz}
For each $i \geq 1$, there are isomorphisms $H^1(G_F,A_i[\pi_i]) \simeq H^1(G_F,A[\pi])$ and $\Sel_{\pi_i}(A_i) \simeq \Sel_\pi(A)$.   
\end{corollary}

Finally, note that the results in this section apply equally well if we replace $A$ and $B$ by $A_d$ and $B_d$ (as well as replacing the isogenies $\phi$ and $\phi'$ by $\phi_d$ and $\phi'_d$), for any $d \in F^\times/F^{\times2}$.   


\section{Local Selmer ratios}\label{sect: local Selmer}
\subsection{Generalities}
For the first part of this section, we let $\phi \colon A \to B$ be any isogeny of abelian varieties over a number field $F$.  We will recall the definitions of the local and global Selmer ratios attached to $\phi$.  One expects (and in certain special cases, one can prove) that these numbers dictate the average behavior of the ranks of the Selmer group $\Sel_\phi(A)$, as $\phi \colon A \to B$ varies through an algebraic family of isogenies.  Assume, for simplicity, that the degree of $\phi$ is $\ell^n$ for some prime $\ell$.  Eventually we will take $\ell = 3$.

For each place $\p$ of $F$, write $F_\p$ for the completion at $\p$.  There is an induced homomorphism of groups $\phi^{(\p)} \colon A_d(F_\p) \to B_d(F_\p)$, and the {\it local} Selmer ratio is defined to be
\begin{equation}\label{eqn: local Selmer}
c_\p(\phi) = \dfrac{\#\mathrm{coker}\, \phi^{(\p)}}{\#\ker \phi^{(\p)}},
\end{equation}
which lies in $\ell^\Z$.  

We also define the {\it global} Selmer ratio $c(\phi) = \prod_\p c_\p(\phi)$, where the product is over all places of $F$, including the archimedean ones.  The following proposition and its corollary show that this is a finite product, and hence the global Selmer ratio is well-defined \cite[Lem.\ 3.8]{Schaefer}.

\begin{proposition}\label{local Selmer}
If $\p$ is a finite place of $F$, then 
\[c_\p(\phi) = \dfrac{c_\p(B)}{c_\p'(A)} \cdot \gamma_{\phi,F_\p},\]
where $c_\p(A)$ is the Tamagawa number of $A$ at $\p$, and $\gamma_{\phi,F_\p}^{-1}$ is the normalized absolute value of the determinant of the Jacobian matrix of partial derivatives of the map induced by $\phi$ on formal groups over $F_\p$, evaluated at the origin.  In particular, $\gamma_{\phi,F_\p} = \ell^k$ for some integer $k \geq 0$.      
\end{proposition}  

\begin{corollary}\label{good reduction}
If $\p$ is a finite place of $F$ not above $\ell$, then $c_\p(\phi) = c_\p(B)/c_\p(A)$.  Hence if $A$ has good reduction at $\p$ and $\ell \nmid \p$, then $c_\p(\phi) = 1$.  
\end{corollary}

The local Selmer ratio at archimedean places of $F$ is easy to compute:
\begin{lemma}\label{R}
If $\p$ is an archimedean place of $F$ and if $\ell \neq 2$, then $c_\p(\phi) = \#A[\phi](F_\p)^{-1}$.  
\end{lemma}

\begin{proof}
In this case, $\mathrm{coker}(\phi^{({F_\p)}})$ is both a 2-group and an $\ell$-group, so is trivial. 
\end{proof}

We will need some formal properties of local and global Selmer ratios.

\begin{lemma}\label{loc mult}
Let $\p$ be a place of $F$ and let $\phi \colon A \to B$ and $\psi \colon B \to C$ be isogenies of abelian varieties over $F_\p$.  Then $c_\p(\psi \circ \phi) = c_\p(\psi)c_\p(\phi)$.
\end{lemma}  

\begin{proof}
\cite[I.7.2]{Milne}.
\end{proof}

\begin{corollary}\label{global mult}
Let $\phi \colon A \to B$ and $\psi \colon B \to C$ be isogenies of abelian varieties over $F$.  Then $c(\psi \circ \phi) = c(\psi)c(\phi)$.  
\end{corollary}

\begin{lemma}\label{comp}
Suppose $\phi \colon A \to B$ and $ \psi \colon B \to C$ are isogenies over $F_p$ for some place $\p$ of $F$.  Then $\gamma_{\psi\circ\phi,F_\p} = \gamma_{\phi,F_\p} \cdot \gamma_{\psi,F_\p}$.  
\end{lemma}

\begin{proof}
This follows from Lemma \ref{loc mult} and Proposition \ref{local Selmer}.
\end{proof}

\begin{lemma}\label{mult}
Let $\O$ be the ring of integers in a totally real field, and suppose $A$ has  $\mathrm{RM}$ by $\O$.   Let $\alpha \in \O$ be an element of norm $\ell^n$ for some $n$.  Suppose also that $\p$ is a place of $F$ dividing $\ell$.  Then $\gamma_{\alpha, F_\p}  = \ell^{n[F_\p \colon \Q_\ell]}$.
\end{lemma}
\begin{proof}
The case $\alpha \in \Z$ is proved in \cite[Prop.\ 3.9]{Schaefer}, and the case of general $\alpha$ is proved in exactly the same way.  Note that it is important here that $\alpha$ is an isogeny, which follows from the fact that the embedding $\O \hookrightarrow \End_F(A)$ sends $1$ to $1_A$.    
\end{proof}


\subsection{RM isogeny chains}
We return to the setup of the previous section, so that $\phi \colon A \to B$ is an $\O$-linear 3-isogeny of abelian varieties with RM by $\O$ over $F$.  We also impose Assumption \ref{assump}.  Recall that $I$ is the kernel ideal of $\phi$ and $\pi \colon A \to C = A/A[I]$ is the natural $(3,3)$-isogeny.  Also recall from Section \ref{chainz} the endomorphism $\e \colon A \to A$ as well as the chain of $(3,3)$-isogenies $\pi_i \colon A_i \to A_{i+1}$ for $i \geq 1$.

Proposition \ref{kernel duality} allows us to compute various local and global Selmer ratios in this setting.  
\begin{lemma}\label{Rpi}
If $\p$ is an archimedean place of $F$, then $c_\p(\pi_i) = \frac13$, for all $i$.  
\end{lemma}

\begin{proof}
It is enough to prove the statement for $\pi_1 = \pi$.  There are only two group schemes of rank 3 over $\R$, namely $\Z/3\Z$ and $\mu_3$, and they are dual to each other.  Thus, by Proposition \ref{kernel duality}, one of $A[\phi]$ and $B[\phi']$ is isomorphic to $\mu_3$ and the other is isomorphic to $\Z/3\Z$ (over $F_\p$).  So $A[\pi]$ is either an extension of $\mu_3$ by $\Z/3\Z$ or an extension of $\Z/3\Z$ by $\mu_3$.  In the former case, we clearly have $\#A[\pi](F_\p) = 3$, which shows that $c_\p(\pi) = \frac13$, by Lemma \ref{R}.  In the latter case, we have 
\[0 \to \mu_3 \to A[\pi] \to \Z/3\Z \to 0\]
over $F_\p \simeq \R$.  Since $H^1(\Gal(\bar \R/\R), \mu_3) \simeq \R^\times/\R^{\times 3} = 0$, we see that $\#A[\pi](F_\p) = 3$, and we again have $c_\p(\pi) = \frac13$.       
\end{proof}

\begin{lemma}\label{localSelmer in chains}
For each $i \geq 1$ and every place $\p$ of $F$, we have $c_\p(\pi_i) = c_\p(\pi)$.  
\end{lemma}
\begin{proof}
This follows from Proposition \ref{isogchain}.  
\end{proof}

\begin{proposition}\label{pi Selmer}
We have $c(\pi) = 1$ and $c(\e) = 1$.  
\end{proposition}

For the proof, we introduce some notation.  For any isogeny $f$ of abelian varieties over $F$, and for any place $p$ of $\Q$, we define $c_p(f) = \prod_{\p \mid p} c_\p(f)$, where the product is over the places of $F$ over $p$.

\begin{proof}
The endomorphism $\epsilon \colon A \to A$ is an isogeny of degree $3^{2k}$.  By Proposition \ref{local Selmer}, we have $c_\p(\epsilon) = 1$ for all finite $\p \nmid 3$.  Thus, by Lemmas \ref{loc mult}, \ref{comp}, \ref{mult}, and \ref{Rpi}, we have 
\[c(\epsilon) = c_3(\e)c_\infty(\e) = 3^{k[F \colon \Q]}3^{-k[F \colon \Q]} = 1.\]
On the other hand, by Lemma \ref{localSelmer in chains}, we have $c(\pi_i)= c(\pi)$ for all $i$.  Thus, 
\[1 = c(\e) = \prod_{i = i}^{k}c(\pi_i) = c(\pi)^k,\]
which shows that $c(\pi) = 1$, as claimed. 
\end{proof}

\begin{remark}{\em
For an alternate proof use Proposition \ref{selfdual} and Theorem \ref{PT duality} below.
}\end{remark}

\begin{corollary}\label{inverse}
We have $ c(\phi') = c(\phi)^{-1}$.  
\end{corollary}

\begin{proof}
This follows from Corollary \ref{global mult} and Proposition \ref{pi Selmer}. 
\end{proof}

\section{Poitou-Tate duality and parity of the $\pi$-Selmer rank}\label{PTSelmer}

An isogeny  $\psi \colon A \to A'$ of abelian varieties over a field $F$, gives rise to a short exact sequence
\[0 \to A[\psi] \to A \to A' \to 0\]
of group schemes over $F$, which we will also view as a short exact sequence of $G_F$-modules.  Taking the long exact sequence in group cohomology gives the Kummer map $A'(F) \to  H^1(F, A[\psi])$.  If $F$ is a number field, then the $\psi$-Selmer group $\Sel_\psi(A)$ is defined to be the subgroup of $H^1(F, A[\psi])$ consisting of cohomology classes whose restriction to $H^1(F_\p, A[\psi])$ lies in the image of the Kummer map $A'(F_\p) \to H^1(F_\p, A[\psi])$, for all places $\p$ of $F$. 

We now place ourselves in the context of the previous sections, so that $A$ is an abelian variety over a totally real field $F$ satisfying Assumption \ref{assump}.  For each $d \in F^\times/F^{\times2}$ we have the 3-isogenies $\phi_d \colon A_d \to B_d$ and $\phi'_d \colon B_d \to C_d$, whose composition is the $(3,3)$-isogeny $\pi_d \colon A_d \to C_d$.  Each of the Selmer groups $\Sel_{\phi_d}(A_d)$, $\Sel_{\phi'_d}(B_d)$ and $\Sel_\pi(A_d)$ is a finite dimensional vector space over $\F_3$.

The purpose of this section is to collect some results which show that the global Selmer ratio $c(\phi_d)$ encodes important information concering the $\F_3$-ranks of these three Selmer groups.  The key input is Poitou-Tate global duality.  In the case of elliptic curves, all results in this section are due to Cassels \cite{Cassels8}.    

\begin{theorem}\label{PT duality}
For $A$ and $B$ as above, we have
\[ c(\phi) = \dfrac{\#\Sel_{\phi}(A)}{\#\Sel_{\phi'}(B)} \cdot \dfrac{\#B[\phi'](F)}{\#A[\phi](F)}.\]
\end{theorem}


\begin{proof}
For each place $\p$ of $F$, the subgroup of local conditions in $H^1(F_\p, B[\phi'])$ defining $\Sel_{\phi'}(B)$ is orthogonal under local Tate duality to the local conditions in $H^1(F_\p, \hat C[\hat \phi'])$ defining $\Sel_{\hat\phi'}(\hat C)$, by \cite[Prop.\ B.1]{CK}.  Thus, by Wiles' duality formula \cite[Thm.\ 8.7.9]{neukirch}, we have
\[c(\hat \phi') = \dfrac{\#\Sel_{\hat\phi'}(\hat C)}{\#\Sel_{\phi'}(B) } \cdot \dfrac{\#B[\phi'](F)}{\#\hat C[\hat\phi'](F)}.\]
On the other hand, by Corollary \ref{isog isom}, we have $\Sel_\phi(A) \simeq \Sel_{\hat\phi'}(\hat C)$, $c(\phi) = c(\hat\phi')$, and $A[\phi] \simeq \hat C[\hat\phi']$, from which we deduce the desired formula.   
\end{proof}

The next result is presumably well-known to experts, but we could not find it in the literature.  This is a generalization of a result Cassels used to prove the non-degeneracy of the Cassels-Tate pairing for elliptic curves.  Recall that the Cassels-Tate pairing for $B$ is a pairing \cite[Prop.\ I.6.9]{Milne}
\[\langle \, , \, \rangle_B \colon \Sha(B) \times \Sha(\hat B) \to \Q/\Z.\]
Using the isogeny $\lambda_B \colon B \to \hat B$  from Corollary \ref{B selfdual}, we construct a bilinear pairing
\[\langle \, , \, \rangle_{\lambda_B} \colon \Sha(B) \times \Sha(B) \to \Q/\Z\]
defined by $\langle a,b \rangle_\lambda = \langle a, \lambda_B(b) \rangle_B$.       
    
\begin{theorem}\label{adjointness}
Suppose $b \in \Sha(B)[\phi']$.  Then $b$ is in the image of $\phi \colon \Sha(A) \to \Sha(B)$ if and only if $\langle b, b'\rangle_{\lambda_B} = 0$ for all $b' \in \Sha(B)[\phi']$.     
\end{theorem}    

\begin{proof}
This is proved exactly as in \cite[Thm.\ 3]{Fisher}, using the fact that $H^1(G_F,A[\phi])$ and $H^1(G_F,B[\phi'])$ are in local duality by Corollary \ref{isog isom}.  See also the proof of \cite[Lem.\ I.6.17]{Milne}.  
\end{proof}   
    
Theorem \ref{adjointness} will be used in the proof of the following theorem.
    
\begin{theorem}\label{parity}
If $c(\phi) = 3^m$, then $\dim_{\F_3} \Sel_\pi(A) \equiv m \pmod 2$.
\end{theorem}

\begin{proof}
A diagram chase gives the following well-known five-term exact sequence
\begin{equation}\label{fiveterm}
0 \to B(F)[\phi'] /\phi(A(F)[\pi]) \to \mathrm{Sel}_{\phi}(A) \to \mathrm{Sel}_\pi(A) \to \mathrm{Sel}_{\phi'}(B) \to \dfrac{\Sha(B)[\phi']}{\phi(\Sha(A)[\pi])} \to 0. 
\end{equation}
Combining this with Theorem \ref{PT duality}, we reduce to showing that 
\[\dim_{\F_3} \dfrac{\Sha(B)[\phi']}{\phi(\Sha(A)[\pi])}  \] 
is even.  But by Theorem \ref{adjointness}, the Cassels-Tate pairing $\langle \, , \, \rangle_{\lambda_B}$ on $\Sha(B)$ restricts to a non-degenerate pairing on this finite group.  Moreover, the pairing is anti-symmetric by \cite[Cor.\ 6]{PoonenStoll}. Since any anti-symmetric pairing on an $\F_3$-vector space is also alternating, it follows from the non-degeneracy that $\Sha(B)[\phi']/\phi(\Sha(A)[\pi])$ has even $\F_3$-rank, as desired.  
\end{proof}

Finally, we note that the theorems in this section apply equally well if we replace $A$ and $B$ by $A_d$ and $B_d$ (as well as replacing the isogenies $\phi$ and $\phi'$ by $\phi_d$ and $\phi'_d$) for any $d \in F^\times/F^{\times2}$.


\section{The ranks of the quadratic twists $A_d$}\label{proofs}

Recall the height function on $d \in F^\times/F^{\times2}$ from the introduction: 
\[H(d) = \prod_{\p \colon \ord_\p(d) \mbox{ is odd}} N(\p),\]
where the product is over the finite primes of $F$.  This gives a way of ordering $F^\times/F^{\times2}$, and lets us define the {\it average value} of a function $f$ defined on $F^\times/F^{\times2}$:
\[\avg \, f(d) =\lim_{X \to \infty} \dfrac{\sum_{H(d) < X}f(d)}{\sum_{H(d) < X} 1}.\]      
We define the {\it density} of a subset $\Sigma \subset F^\times/F^{\times2}$ to be 
\[\mu(\Sigma) = \avg \, 1_\Sigma(d),\] 
where $1_\Sigma$ is the characteristic function of $\Sigma$.  

We say $\Sigma$ is defined by {\it finitely many local conditions} if there exists for each place $\p$ of $F$ subsets $\Sigma_\p \subset F_\p^\times/F_\p^{\times 2}$, such that $\Sigma_\p = F_\p^\times/F_\p^{\times 2}$ for all but finitely many $\p$ and such that 
\[\Sigma = F^\times/F^{\times 2} \cap \prod_\p \Sigma_p,\] with the intersection taking place inside $\prod_\p F_\p^\times/F_\p^{\times 2}$.  The average of a function $f$ on $\Sigma$ (ordered by the height function $H$) is denoted $\avg_\Sigma \, f(d)$.  

The key input for our proofs of Theorems \ref{main intro} and \ref{rank g} is the following general result of Bhargava, Klagsbrun, Lemke Oliver, and the author \cite[Thm.\ 1.1]{BKLS}.

\begin{theorem}\label{bkls}
Suppose $\psi \colon A \to A'$ is a degree $3$ isogeny of abelian varieties over a number field $F$, and let $\psi_d \colon A_d \to A'_d$ be the corresponding quadratic twist family of $3$-isogenies, for $d \in F^\times/F^{\times2}$.  If $\Sigma \subset F^\times/F^{\times 2}$ is a subset defined by finitely many local conditions, then the average size of $\Sel_{\psi_d}(A_d)$, for $d \in \Sigma$ is $1 + \avg_\Sigma \, c(\psi_d)$.  
\end{theorem}  

Now assume that $A$ is abelian variety over a totally real field $F$ satisfying Assumption \ref{assump}. 
For $m \in \Z$, we define the subsets 
\[T_m(\phi) = \left\{d \in F^\times/F^{\times2} \colon c(\phi_d)  = 3^m \right\} \subset F^\times/F^{\times 2}.\]
Then the sets $T_m(\phi)$ are defined by finitely many local conditions, and for any fixed $m$, either $T_{m}(\phi)$ is empty or it has positive density.  We also write $T_{\pm m}(\phi)$ for $T_m(\phi) \cup T_{-m}(\phi)$.  

Our first result gives a concrete bound on the average rank of the quadratic twists $A_d$.  To state the result cleanly, we define the {\it absolute log-Selmer ratio} $t(\phi_d) = |\ord_3 \, c(\phi_d)|$.  

\begin{theorem}\label{rank bound}
Suppose $A$ has dimension $g$.  Then the average Mordell--Weil rank of the quadratic twists $A_d$, for $d \in F^\times/F^{\times 2}$, is at most $g \cdot \avg_d \left(t(\phi_d) + 3^{t(\phi_d)} \right)$.      
\end{theorem}

\begin{proof}
Note that $A_d(F)$ is a finitely generated $\O$-module, and that $A_d(F)$ has no 3-torsion for all but finitely many $d$.  For such $d$, the rank of $A_d(F)$ is equal to $ (g/k) \cdot \rk_{\O/\e} \left(A_d(F)/\e A_d(F)\right)$.  Recall that $k$ is the order of the ideal class of $I$ in $\Pic(\O)$, and $\e$ is the generator of the $k$th power of the kernel ideal $I$ of $\phi$.  
We have $A_d(F)/\e A_d(F) \subset \Sel_\e(A_d)$, so it is enough to show that the average $\O/e$-module rank of $\Sel_\e(A_d)$ is at most $k\cdot \avg_d \left(t(\phi_d) + 3^{t(\phi_d)} \right)$.  Since $\e = \pi_k\pi_{k-1} \cdots \pi_2 \pi_1$, and since $\Sel_{\pi_i}(A_i) \simeq \Sel_{\pi}(A)$ for all $i$ (Corollary \ref{selmerchainz}), the rank of $\Sel_\e(A_d)$ as an $\O/e$-module is at most $k \dim_{\F_3}\Sel_\pi(A)$.   So it is enough to show that the average size of $\dim_{\F_3}\Sel_\pi(A_d)$ is at most $\avg_d \left(t(\phi_d) + 3^{t(\phi_d)} \right)$

To prove this, we fix $m \in \Z$ and show that the average $\F_3$-rank of $\Sel_\pi(A_d)$ for $d \in T_{m}(\phi)$ is at most $|m| + 3^{|m|}$.  We will suppose $m \geq 0$; the proof in the case $m < 0$ is similar.  Then by Theorem \ref{bkls}, the average size of $\Sel_\phi(A_d)$ for $d \in T_m(\phi)$ is $1 + 3^m$.  For any $r \in \Z$, we have the inequality $2r + 1 - 2m \leq 3^{r-m}$, so it follows that the average $\F_3$-rank of $\Sel_\phi(A_d)$ for $d \in T_m(\phi)$ is at most $m + \frac123^{-m}$.  By Corollary \ref{inverse}, the average size of $\Sel_{\phi'}(B_d)$ for $d \in T_m(\phi)$ is $1 + 3^{-m}$, so the average $\F_3$-rank of $\Sel_{\phi'}(B_d)$ is at most $\frac12 3^{-m}$.           

Using the exact sequence from (\ref{fiveterm}),
\begin{equation}\label{3term}
\Sel_{\phi_d}(A_d) \to \Sel_\pi(A_d) \to \Sel_{\phi'_d}(B_d),
\end{equation}
 we deduce that the average $\F_3$-rank of $\Sel_\pi(A_d)$ is bounded by $ m + 3^{-m},$
 as desired.  
\end{proof}

To show that a positive proportion of quadratic twists $A_d$ have $\pi$-Selmer ranks 0 (resp.\ 1), we will need the following proposition.  

\begin{proposition}\label{positive density}
The sets $T_0(\phi)$ and $T_{\pm 1}(\phi)$ have positive density in $F^\times/F^{\times2}$.  
\end{proposition}       

\begin{proof}
We will prove the theorem for $T_0(\phi)$; the proof for $T_{\pm 1}(\phi)$ is similar.  We need to construct a set $\Sigma \subset F^\times/F^{\times 2}$, defined by finitely many local conditions, such that $c(\phi_d) = 1$ for all $d \in \Sigma$.  In fact, we construct $\Sigma$ such that $c_\p(\phi_d) = 1$, for all $\p \nmid 3\infty$, and for all $d \in \Sigma$.  

Let $N_A$ be the conductor ideal of $A$.  Then $\p$ divides $N_A$ if and only if $A$ has bad reduction at $\p$, and  $c_\p(\phi_d) = c_\p(\phi'_d) = 1$ for all primes $\p$ not dividing $6N_A\infty$, and for all $d$, by \cite[Thm.\ 5.2]{BKLS}. To handle primes $\p$ dividing $6N_A$, first note that there is exactly one squareclass $d\in F_\p^\times/F_\p^{\times 2}$ such that the denominator in (\ref{eqn: local Selmer}), for $\phi = \phi_d$, is not equal to 1.  Since there are at least four squareclasses in $F_\p^\times$, we may choose a set $\tilde \Sigma \subset F^\times/F^{\times 2}$, defined by finitely many congruence conditions, such that both $c_\p(\phi_d)$ and $c_\p(\phi'_d)$ are integers, for all finite primes $\p$ and all $d \in \tilde\Sigma$.  

This already implies that $c_\p(\phi_d) = 1 = c_\p(\phi'_d)$ for all finite primes  $\p \nmid 3$ and all $d \in \tilde\Sigma$.  Indeed, for such $\p$ and $d$, we have $c_\p(\pi_i) = c_\p(\pi) = c_\p(\phi_d) c_\p(\phi'_d)$, so that the ratios $c_\p(\pi_i)$ are integers whose product is 1: 
\[c_\p(\pi_1)c_\p(\pi_2) \cdots c_\p(\pi_k) = c_\p(\e) = c_\p(A_d)/c_\p(A_d) = 1.\]
Here we have used Proposition \ref{local Selmer}.  It follows that both $c_\p(\phi_d)$ and $c_\p(\phi'_d)$ are positive integers whose product is 1, and hence both of these local Selmer ratios are 1, as claimed.      
 
 For the computation at primes $\p \mid 3$, we abbreviate $\gamma_d = \gamma_{\phi_d, F_\p}$ and $\gamma'_d = \gamma_{\phi'_d, F_\p}$.  Then
 \[\prod_{i = 1}^kc_\p(\pi_i) = \prod_{i = 1}^k \gamma_{\pi_i} = \gamma_\e = 3^{k[F_\p \colon \Q_3]},\] for all $d$, by Proposition \ref{local Selmer}, Lemma \ref{comp}, and Lemma \ref{mult}.  Since $c_\p(\pi_{i}) = c_\p(\pi)$, we have that $c_\p(\pi) = c_\p(\phi_d)c_\p(\phi'_d) = 3^{[F_\p \colon \Q_3]}$. Writing 
 \[c_3(\phi_d) := \prod_{\p \mid 3} c_\p(\phi_d) \hspace{3mm} \mbox{ and } \hspace{3mm} c_3(\phi'_d) := \prod_{\p \mid 3} c_\p(\phi'_d),\] 
 we conclude that for $d \in \tilde \Sigma$, the product $c_3(\phi_d)c_3(\phi'_d)$ is a power of 3 satisfying
 \[1 \leq c_3(\phi_d) c_3(\phi'_d) \leq 3^{[F \colon \Q]}. \] 
In particular, $ 1\leq c_3(\phi_d) \leq 3^{[F \colon \Q]}$ for all $d \in \tilde \Sigma$.   
 
Finally, if $\p$ is a (real) archimedean place of $F$,  then $c_\p(\phi_d)$ is equal to 1/3 or 1, depending on whether or not $A[\phi_d](F_\p)$ is non-trivial.  Thus, we can impose sign conditions on $d$ so that $c_\infty(\phi_d) := \prod_{\p \mid \infty} c_\p(\phi_d)$ is equal to $3^{-n}$ for any $n$ satisfying $1 \leq n \leq [F \colon \Q]$.

It follows from the above discussion that the subset $\Sigma \subset \tilde \Sigma$ defined by 
 \[ \Sigma = \left\{d \in \tilde \Sigma \colon c_3(\phi_d)c_\infty(\phi_d) = 1\right\},\]
 has positive density in $F^\times/F^{\times 2}$.  Moreover, $c(\phi_d) = 1$ for all $d \in \Sigma$, as desired.
\end{proof}

\begin{theorem}\label{main}
A proportion of at least $\frac12\mu(T_0(\phi))$ of the quadratic twists $A_d$ have rank $0$.    
\end{theorem}

\begin{proof}
Since $A_d(F)/\e A_d(F) \subset \Sel_\e(A_d)$, if $\Sel_\e(A_d)$ has $\F_3$-rank 0, then $A_d(F)$ has rank 0.  So it is enough to prove that at least $\frac12\mu(T_0(\phi))$ of twists satisfy $\Sel_\e(A_d) = 0$.  Since $\Sel_{\pi_i}(A_i) \simeq \Sel_{\pi}(A)$, and $\e = \pi_k\cdots\pi_1$, it is enough to prove that at least $\frac12\mu(T_0(\phi))$ of twists satisfy $\Sel_\pi(A_d) = 0$.  We may of course assume $T_0(\phi)$ is non-empty, otherwise there is nothing to prove. 

By Theorem \ref{bkls}, the average size of $\Sel_\phi(A_d)$ for $d \in T_0(\phi)$ is $2$.  Since $\Sel_\phi(A_d)$ is an $\F_3$-vector space, it follows that at least $50\%$ of $d$ in $T_0(\phi)$ are such that $\Sel_\phi(A_d) = 0$.  By Theorem \ref{PT duality}, $100\%$ of these $d$ satisfy $\Sel_{\phi'_d}(A_d) = 0$, as well.  From the exact sequence (\ref{3term}) we see that $\Sel_\pi(A_d) = 0$ for all such $d$, proving the theorem.
\end{proof}

Theorem \ref{main intro} now follows immediately from Proposition \ref{positive density} and Theorem \ref{main}.  To prove Theorem \ref{rank g} we use the following result.

\begin{theorem}\label{rank1}
A proportion of at least $\frac56\mu(T_{\pm 1}(\phi))$ of the $A_d$ satisfy $\dim_{\F_3} \Sel_\pi(A_d) = 1$.    
\end{theorem}

\begin{proof}
We first consider $T_1(\phi)$.  As before, we may assume this set has positive density.  By Theorem \ref{bkls}, the average size of $\#\Sel_\phi(A_d)$ on $T_1(\phi)$ is $4$.  On the other hand, by Theorem \ref{PT duality}, the $\F_3$-dimension of $\Sel_\phi(A_d)$ is at least 1 for $100\%$ of twists $d \in T_1(\phi)$.  It follows that for at least $\frac56$ of such $d$, we have $\#\Sel_\phi(A_d) = 3$ and (by Theorem \ref{PT duality}) that $\#\Sel_{\phi'}(B_d) = 1$.  By the exactness of (\ref{3term}), we conclude that $\dim_{\F_3} \Sel_\pi(A_d) = 1$ for such $d$.  

Next we consider $T_{-1}(\phi)$.  By \cite[Thm.\ 1.1]{BKLS}, the average size of $\Sel_\phi(A_d)$ for $d \in T_{-1}(\phi)$ is $\frac43$.  It follows that at least $\frac56$ of $d \in T_{-1}(\phi)$ satisfy $\#\Sel_\phi(A_d) = 1$ and (by Theorem \ref{PT duality}) that $\#\Sel_{\phi'}(B_d) = 3$.  For such $d$, the $\F_3$-dimension of $\Sel_\pi(A_d)$ is either 0 or 1.  But by Theorem \ref{parity}, the $\F_3$-dimension of $\Sel_\pi(A_d)$ is odd, so $\Sel_\pi(A_d)$ has $\F_3$-rank 1 for at least $\frac56$ of $d$ in $T_{-1}(\phi)$.                     
\end{proof}

\begin{proof}[Proof of Theorem $\ref{rank g}$]
The first part of Theorem \ref{rank g} follows immediately from Proposition \ref{positive density} and Theorem \ref{rank1}.  For the second part, first note that there are isomorphisms
\[C_d(F)/\pi A_d(F) \simeq A_{i+1,d}(F)/\pi_i A_{i,d}(F)\]
for all $1 \leq i \leq k$ and all $d \in F^\times/F^{\times2}$, by Proposition \ref{isogchain}.  Thus, if $A_d(F)$ has no 3-torsion (which is true for almost all $d$) and has rank $ng$, then $A_d(F)/\e A_d(F)$ has rank $n$ over $\O/\e$ and $C_d(F)/\pi A_d(F)$ has $\F_3$-dimension $n$.  We also have the short exact sequence
\[0 \to C_d(F)/\pi A_d(F) \to \Sel_\pi(A_d) \to \Sha(A_d)[\pi] \to 0.\]
Thus, if $\dim_{ \F_3} \Sel_\pi(A_d) = 1$, then either $A_d$ has rank $g$ or $\dim_{\F_3} \Sha(A_d)[\pi] = 1$.  But if $\Sha(A_d)$ is finite, then $\Sha(A_d)[\pi]$ has square order \cite[Thm.\ 3.3]{Chao}, so the latter case cannot occur.         
\end{proof}



\section{Optimal quotients of prime conductor}

In this section we prove Theorem \ref{mazur}.  The main ingredient is Theorem \ref{main intro}, but we also need Mazur's results in \cite{Mazur} to verify that the relevant kernel ideal is invertible.  We first set notation and terminology.  

Fix a prime $p$ and let $\T$ be the $\Z$-algebra of Hecke operators acting on the space of weight 2 cuspforms on $\Gamma_0(p)$.  If $f$ is a newform for $\Gamma_0(p)$, then there is an induced homomorphism $\T\to \C$ giving the action of Hecke operators on $f$.  Let $I_f$ denote the kernel of this homomorphism.  The abelian variety $A = J_0(p)/I_fJ_0(p)$ is the {\it optimal quotient} corresponding to $f$.  The endomorphism algebra $\End_\Q(A) \otimes \Q$ is isomorphic to $\O \otimes_\Z \Q$, where $\O := \T/I_f$.  Mazur proved \cite[Prop.\ II.9.5]{Mazur} that $\T = \End_\Q(J_0(p))$, from which it follows that $\T/I_f$ is a subring of $\End_\Q(A)$ of finite index, though not necessarily equal to $\End_\Q(A)$.  In any case, $A$ has RM by $\O$.    

If $p \equiv 1 \pmod 9$, then there is at least one optimal quotient $A$ with a rational point $P$ of order 3.  If moreover $p \not\equiv 1 \pmod{27}$, then there is exactly one such quotient, corresponding to a single Galois orbit of newforms.  These facts follow from a theorem of Emerton \cite[Thm.\ B]{Emerton}.    

\begin{proof}[Proof of Theorem $\ref{mazur}$]
By \cite[Thm.\ 4.13]{Emerton}, the ring $\O$ necessarily contains an ideal $I$ of index 3 such that $P$ is contained in $A[I]$.  Specifically, if $\mathfrak{I} \subset \T$ is the Eisenstein ideal, then $I$ is the image of $(3,\mathfrak{I})$ under the map $\T \to \T/I_f$.  By Remark \ref{3isog}, the quotient $\phi \colon A \to A/\langle P\rangle$ is an $\O$-linear 3-isogeny with kernel ideal $I$. Since we work over $\Q$, every polarization of $A$ is $\O$-linear as well.  

To prove Theorem \ref{mazur}, it remains to check the final hypothesis of Theorem \ref{main intro}, that $I$ is an invertible $\O$-module.  It is enough to check that $I$ is a locally free $\O$-module, i.e.\ that $I_\p \simeq \O_\p$ as $\O_\p$-modules, for all primes $\p$ of $\O$. Since $I$ has index 3 in $\O$, it is itself a prime ideal, and $I_I \simeq \O_I$ by Mazur's principality theorem \cite[Thm.\ II.18.10]{Mazur}.    Here it is crucial that $p \not\equiv 1 \pmod{27}$, so that the image of $(3, \mathfrak{I})$ in $\O$ is all of $I$, and not some smaller ideal.  On the other hand, if $\p \neq I$, then the inclusion $I \hookrightarrow \O$ induces an isomorphism $I_\p \simeq \O_\p$.                   
\end{proof}


\section{Examples}\label{examples}
We give explicit examples of abelian varieties of dimension $g > 1$ over $\Q$ satisfying the hypotheses of Theorems \ref{main intro} and \ref{rank g}, and Corollary \ref{curves}.  

The first few examples were found by searching through rational points on explicit models of Hilbert modular surfaces.  These are moduli spaces of principally polarized abelian surfaces with real multiplication by a fixed order $\O$ in a real quadratic number field $K$.  For our purposes, we must restrict to those $K$ with an ideal of index 3.  If we take $K = \Q(\sqrt{3p})$ with $p$ equal to 1, 2, or any prime $p \equiv 3 \pmod 4$, then we can take the ideal to be principal as well.  In those cases, we have $3\O_K = \pi^2 \O_K$ for some element $\pi \in \O_K$ of norm $\pm3$.  It follows that if $A/\Q$ has RM by $\O_K$, then $A(\Q)$ has a non-trivial $\pi$-torsion point if and only if $A(\Q)$ has a point of order 3.  

To find examples of such $A$, we simply searched through the tables of genus two curves in \cite{Elkies-Kumar}, while checking in Magma for a rational point of order 3 on the Jacobian.  

\begin{example}{\em 
The Jacobian $J$ of the genus two curve 
\[C \colon y^2 = x^5 - x^4 + x^3 - 3x^2 - x + 5\] has RM by $\Z[\sqrt{3}]$ over $\Q$ and a non-trivial $\sqrt{3}$-torsion point in $J(\Q)$.  Thus Theorems \ref{main intro} and \ref{rank g} apply.  As $C$ is hyperelliptic, Corollary \ref{curves} applies as well, and shows that for a positive proportion of twists $d$, we have $C_d(\Q) = \{\infty, (-1,0)\}$.  
}
\end{example}

\begin{example}{\em 
The Jacobian $J$ of the genus two curve 
\[C \colon y^2 = -72x^6 + 84x^5 + 127x^4 - 123x^3 - 83x^2 + 51x + 25\] has RM by $\Z[\sqrt6]$ over $\Q$ and a non-trivial $(3 + \sqrt{6})$-torsion point $P$ in $J(\Q)$.  In this case, we have $B = J/\langle P \rangle$, which is itself principally polarized (see Remark \ref{pp}) and hence a Jacobian of some genus 2 curve $C'$.  If one could write down a model for $C'$, then the bound on the average rank of $J_d$ given in Theorem \ref{rank bound} could be computed explicitly, and our lower bounds on the proportion of twists having rank 0 (resp.\ $\pi$-Selmer rank 1) could be made explicit as well.     
}\end{example}

We found the following example in the LMFDB \cite{lmfdb}.
\begin{example}\label{example 65}{\em 
The Jacobian $J$ of the genus two curve 
\[C \colon y^2 + (x^3 + 1)y = x^5 + 4x^4 + 6x^3 + 10x^2 + 3x + 1\] has RM by $\Z[\sqrt3]$ over $\Q$ and a non-trivial $\sqrt{3}$-torsion point in $J(\Q)$.  This $J$ seems to be isogenous to another Jacobian which is a quotient of $J_0(65)$, and which has appeared in the literature a few times already; see \cite[\S4.2]{GGR} and \cite[Rem.\ 3.4]{Chao}.     
}\end{example}

We have examples in dimension $g > 2$ as well, obtained as optimal quotients of modular Jacobians.   These abelian varieties correspond to Galois orbits of newforms for $\Gamma_0(N)$.  
One can do explicit computations with these abelian varieties, and in particular, one can compute the degree of a polarization of induced by the principal polarization on $J_0(N)$.  The computations in the following examples were performed in sage.         

\begin{example}{\em
Let $A$ be the `minus part' of the modular Jacobian $J_0(127)$.  This is a geometrically simple 7-dimensional abelian variety corresponding to the unique Galois orbit of newforms of level $127$ with root number $+1$.  In particular, it is the only Eisenstein optimal quotient of $J_0(127)$, hence has a torsion point of order 3.  
The modular degree of $A$ is 8, and so $A$ has a polarization of degree prime to 3.  It follows from Theorem \ref{mazur} that a positive proportion of twists $A_d$ have rank 0, and assuming $\Sha(A_d)$ finite, a positive proportion of twists have rank 7.  }
\end{example}   

In the next example, Theorem \ref{mazur} does not apply since $p \equiv 1 \pmod{27}$, but we can verify the hypotheses of Theorem \ref{main intro} nonetheless. 

\begin{example}{\em
There is a unique Eisenstein optimal quotient $A$ of $J_0(109)$; it is 4-dimensional. We have $A(\Q) \simeq \Z/9\Z$ and $\End(A) \simeq \O_K$, for the quartic field $K$ with minimal polynomial $	
x^4 - 5x^3 + 3x^2 + 6x + 1$.  One checks that the order 3 subgroup $G$ of $A(\Q)$ is killed by the unique ideal $I$ of index 3 in $\O_K$.  It follows that $A \to A/G$ is $\O_K$-linear with invertible kernel ideal, and hence that $A$ has a polarization of degree prime to $3$.  This is confirmed in sage, which reports that the modular degree of $A$ is 32. By Theorem \ref{main intro}, a positive proportion of twists $A_d$ have rank 0 and, by Theorem \ref{rank g}, a positive proportion of twists have rank 4, assuming $\Sha(A_d)$ is finite.   
}
\end{example}

One finds other examples in levels $163$ and 181 of dimensions 7 and 9, respectively.

\bibliographystyle{abbrv}
\bibliography{refs}

\end{document}